\documentclass{amsart}
\usepackage[english]{babel}
\usepackage{amssymb,graphicx,enumerate,xcolor}
\usepackage{fancyhdr}
\usepackage{hyperref}
\usepackage[style=numeric,maxnames=10]{biblatex}
\usepackage{csquotes}
\bibliography{ref}
\renewbibmacro{in:}{}

\newcommand{\assign}{:=}
\newcommand{\backassign}{=:}
\newcommand{\cdummy}{\cdot}
\newcommand{\comma}{{,}}
\newcommand{\infixand}{\text{ and }}
\newcommand{\mathe}{\mathrm{e}}
\newcommand{\nobracket}{}
\newcommand{\nocomma}{}
\newcommand{\nosymbol}{}
\newcommand{\tmcolor}[2]{{\color{#1}{#2}}}
\newcommand{\tmem}[1]{{\em #1\/}}
\newcommand{\tmop}[1]{\ensuremath{\operatorname{#1}}}
\newcommand{\tmstrong}[1]{\textbf{#1}}
\newcommand{\tmtextit}[1]{\text{{\itshape{#1}}}}
\newenvironment{enumerateroman}{\begin{enumerate}[i.] }{\end{enumerate}}
\newcounter{nnacknowledgments}

{\theoremstyle{remark}\newtheorem{acknowledgments*}[nnacknowledgments]{Acknowledgments}}
\newtheorem{corollary}{Corollary}
\newtheorem{definition}{Definition}
\newtheorem{lemma}{Lemma}
\newtheorem{proposition}{Proposition}
{\theoremstyle{remark}\newtheorem{remark}{Remark}}
\newtheorem{theorem}{Theorem}

\begin{document}
	
	\pagestyle{fancy}
	\fancyhead{}
	\fancyhead[C]{A Constrained Mean Curvature Flow}
	\fancyhead[R]{\thepage}
        \setlength{\headheight}{12.0pt}

\title{A Constrained Mean Curvature Flow On Capillary Hypersurface Supported on totally geodesic plane}

\author{Xiaoxiang Chai}
\address{Department of Mathematics, POSTECH, Pohang, Gyeongbuk, South Korea}
\email{xxchai@kias.re.kr, xxchai@postech.ac.kr}

\author{Yimin Chen}
\address{Department of Mathematics, Pusan National University, Busan, South
Korea}
\email{sherlockpoe@pusan.ac.kr}

\begin{abstract}
  We prove a new Minkowski type formula for capillary hypersurfaces supported
  on totally geodesic hyperplanes in hyperbolic space. It leads to a
  volume-preserving flow starting from a star-shaped initial hypersurface. We
  prove the long-time existence of the flow and its uniform convergence to a
  $\theta$-totally umbilical cap. Additionally, we establish that a
  $\theta$-totally umbilical cap is an energy minimizer for a given enclosed volume.
\end{abstract}

{\maketitle}

\section{Introduction}

\

Mean curvature flow has a rich history, dating back to significant works such
as Huisken {\cite{huisken1984flow}}. Huisken showed that a convex and closed
hypersurface will flow to a sphere under the properly rescaled mean curvature
flow. A constrained curvature flow is a flow which preserves some geometric
quantities. In $\mathbb{R}^2$, Gage \ {\cite{gage1986area}} used a constrained
curve shortening flow to prove an isoperimetric inequality. In higher
dimensions, a constrained mean curvature flow was applied to prove the
isoperimetric inequality by Huisken {\cite{huisken1987volume}}. This
constrained flow preserves the enclosed volume while decreasing the area of
the hypersurface.

An alternative approach to create a flow that preserves the enclosed volume is
by employing the Minkowski formula on the hypersurface. This has been explored
in {\cite{guan2015mean}}, which investigated such flows in space forms.
Furthermore, this approach has been extended to warped product spaces in
{\cite{guan2019volume}}, where they considered the flow $x : (\mathbb{R}
\times N, g) \rightarrow (\bar{M}, \bar{g} = d r^2 + \phi^2 (r) g_N)$
satisfying
\[ \frac{\partial x}{\partial t} = (n \phi' - u H) \nu . \]
Under appropriate assumptions on the metric $\bar{g}$, the flow is expected to
converge to a level set of $\phi$. See for example {\cite{andrews2001volume}},
{\cite{andrews2018quermassintegral}}, {\cite{andrews2021volume}},
{\cite{hu2022geometric}} and {\cite{hu2022locally}} for various types of fully
nonlinear curvature flow and anisotropic curvature flow in different ambient
spaces.

In recent years, there has been considerable interest in geometric flows of
capillary hypersurfaces, for instance, inverse mean curvature flow with free
boundary in the Euclidean unit ball {\cite{lambert2016inverse}}.

Diving into the main topic of this paper, constrained curvature flow on
capillary hypersurfaces has yielded significant results. These include:
constrained inverse mean curvature type {\cite{scheuer2022alexandrov}},
{\cite{weng2022alexandrov}} and mean curvature type flow {\cite{hu2023mean}},
{\cite{wang2020mean}} for capillary hypersurfaces in the Euclidean unit ball;
curvature flows {\cite{hu2022complete}}, {\cite{mei2024constrained}} and
{\cite{wang2023alexandrov}} in capillary hypersurfaces in Euclidean half
space; \ mean curvature type flow {\cite{mei2023constrained}} in geodesic ball
in space forms.

In this paper, we consider a new constrained mean curvature type flow for
capillary hypersurfaces, which are supported on totally geodesic hyperplanes
in hyperbolic space $\mathbb{H}^{n + 1}$. We use the well-known
Poincar$\acute{\mathe}$ half space model $(\mathbb{H}^{n + 1}, \langle
\cdummy, \cdummy \rangle) = \left( \mathbb{R}^{n + 1}_+, \frac{1}{x^2_{n + 1}}
\langle \cdummy, \cdummy \rangle_{\delta} \right)$, where $x_{n + 1}$ is the
$(n + 1)$-th coordinate, $\langle \cdummy, \cdummy \rangle_{\delta}$ is the
Euclidean metric and $\mathbb{R}^{n + 1}_+ : = \{ x \in \mathbb{R}^{n + 1} :
x_{n + 1} > 0 \}$. Let $P : = \{ x \in \mathbb{H}^{n + 1} : x_1 = 0 \}$ be a
totally geodesic hyperplane. Denote by $x$ the position vector in
$\mathbb{R}^{n + 1}_+ $and $\{ E_i \}_{i = 1}^{n + 1}$ the coordinate basis of
$(\mathbb{R}^{n + 1}, \langle \cdummy, \cdummy \rangle_{\delta})$.

Throughout the paper, we consider $x_0 : M \rightarrow \mathbb{H}^{n + 1}$ as
an immersion of hypersurface in $\mathbb{H}^{n + 1}$. If $\Sigma = x (M)$
satisfies that
\begin{enumerateroman}
  \item $\tmop{int} \Sigma = x (\tmop{int} M) \subset P_+ \assign \{ x \in
  \mathbb{H}^{n + 1} : x_1 > 0 \}$,
  
  \item $\partial \Sigma = x (\partial M) \subset \{ x  \in \mathbb{H}^{n + 1}
  : x_1 = 0 \} = P$,
  
  \item $\Sigma$ and $P$ contacts at a constant angle $\theta$ on $\partial
  \Sigma = \Sigma \cap P$,
\end{enumerateroman}
we call $\Sigma$ a \tmtextit{$\theta$-capillary hypersurface} supported on
$P$, and $P$ the supporting hypersurface.

On a $\theta$-capillary hypersurface supported on totally geodesic hyperplane
$P$, we introduce a novel condition called {\tmem{star-shapedness with respect
to $c E_{n + 1}$}}.

\begin{definition}
  Let $x : M \rightarrow P_+ \subset \mathbb{H}^{n + 1}$ be a
  $\theta$-capillary hypersurface supported on $P$. We say $\Sigma$ is
  star-shaped with respect to $c E_{n + 1}$ if it satisfies that
  \[ \langle x - c E_{n + 1}, \nu \rangle > 0. \]
\end{definition}

We consider a flow, defined as a family of embeddings $x : M \times [0, T)
\rightarrow \mathbb{H}^{n + 1}$ with $x (\partial \Sigma, \cdummy) \subset P$,
such that
\begin{equation}
  \begin{array}{rclccl}
    (\partial_t x)^{\bot} & = & q_c \nu, &  &  & \tmop{in} M \times
    [0, T) ;\\
    \langle \nu, \bar{N} \circ x \rangle & = & - \cos \theta, &  &  &
    \tmop{on} \partial M \times [0, T) ;\\
    x (\cdummy, 0) & = & x_0 (\cdummy), &  &  & \tmop{on} M,
  \end{array} \label{flow}
\end{equation}
where the normal velocity $q_{c}$ is given by
\[ q_c = \frac{n c}{x_{n + 1}} - n c \cos \theta \langle E_1, \nu
   \rangle - H \langle x - c E_{n + 1}, \nu \rangle . \]

We denote by $\kappa$ the principal curvature of an umbilical hypersurface
$\mathcal{C}$. Umbilical hypersurfaces in hyperbolic space can be classified
into three types depending on $\kappa$ as depicted in the figure. In the case
of $\kappa = 0$, it is a totally geodesic hyperplane; in the case of $\kappa >
1$, it is a geodesic sphere, and for $0 < \kappa < 1$, it is an equidistant
hypersurface, and if $\kappa = 1$, it is a horosphere. In the Poincar{\'e}
half space model, $\mathcal{C}$ can be represented as a plane or sphere with
respect to the Euclidean metric. It is easy to see that compact umbilical
$\theta$-capillary hypersurface, can be part of a geodesic sphere, a
horosphere and an equidistant hypersurface.

\begin{figure}[ht]
\includegraphics[width=10cm]{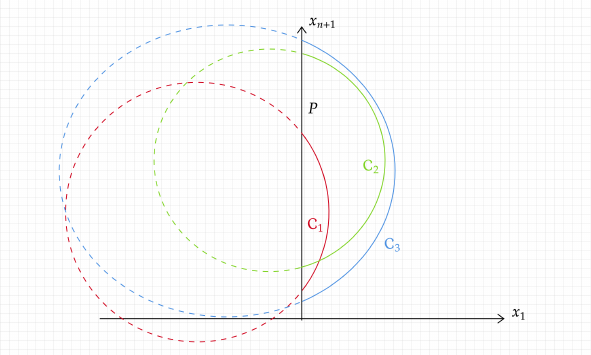}
 \caption{$\mathcal{C}_1$, equidistant hypersurface; $\mathcal{C}_2$, geodesic
 ball; $\mathcal{C}_3$, horosphere.}
\end{figure}

To state our main theorem, we need the following definition.

\begin{definition}
  We define the $\theta$-umbilical cap $\mathcal{C}_{c, R, \theta}$ as follows
  \begin{equation}
    \mathcal{C}_{c, R, \theta} (a) = \{ x \in \mathbb{H}^{n + 1} : |x - R \cos
    \theta E_1 - a - c E_{n + 1} |_{\delta} \leq R, x_1 \geq 0 \},
    \label{capbarrier}
  \end{equation}
  where $a$ is a constant vector perpendicular to both $E_1$ and $E_{n + 1}$. 
\end{definition}

For a $\theta$-umbilical cap, we define a constant $K_0 (c, R, \theta)$ by
\[ K_0 (c, R, \theta) = \left\{\begin{array}{llll}
     c - R \sin \theta &  &  & \theta \leq \pi / 2\\
     c - R &  &  & \theta > \pi / 2
   \end{array}\right. . \]
Note that $K_0 (c, R, \theta) > 0$ if and only if the $\mathcal{C}_{c, R,
\theta} (a)$ is compact. Combining with Remark \ref{capstatic}, we know that
its principal curvature $\kappa = \frac{c}{R} > \sin \theta$.

\

Now we are ready to state our main theorem.

\begin{theorem}
  \label{thm}Let $x_0 : M  \rightarrow P_+ \subset \mathbb{H}^{n + 1}$ be an
  embedding of a compact capillary hypersurface $\Sigma_0 = x_0 (M)$,
  supported on the totally geodesic plane $P$ with constant contact angle
  $\theta$. Suppose there exist constants c, R such that $K_0 (c, R, \theta) >
  c (n - 1) / 4 n$, and \ $\Sigma_0$ is contained in the cap $\mathcal{C}_{c,
  R, \theta} (a)$ and star-shaped with respect to $c E_{n + 1}$. We assume
  that in addition, $\theta$ satisfies that
  \begin{equation}
    | \cos \theta | < \frac{4 n K_0 (c, R, \theta) - c (n - 1)}{4 n K_0 (c, R,
    \theta) + c (n - 1)} . \label{anglecond}
  \end{equation}
  Then the flow \eqref{flow} exists globally with uniform
  $C^{\infty}$-estimates. Moreover, $x (\cdummy, t)$ uniformly converges to an
  umbilical cap in $C^{\infty}$ topology as $t \rightarrow \infty$, with the
  same volume of its enclosed domain as $\Sigma_0$. 
\end{theorem}

\begin{remark}
  In {\cite{mei2023constrained}} and {\cite{wang2020mean}}, the angle
  condition $| \cos \theta | \leq \frac{3 n + 1}{5 n - 1}$ is required similar
  to our angle condition \eqref{anglecond}. By rearranging the terms, we find
  that the condition \eqref{anglecond} is equivalent to
  \[ | \cos \theta | < \frac{3 n + 1 - 4 n \sin \theta R / c}{5 n - 1 - 4 n
     \sin \theta R / c}, \quad \tmop{if} \hspace{1.2em} \theta \leq \pi / 2 ;
  \]
  and
  \[ | \cos \theta | < \frac{3 n + 1 - 4 n R / c}{5 n - 1 - 4 n R / c}, \quad
     \tmop{if} \hspace{1.2em} \theta > \pi / 2. \]
  Compared to the condition in {\cite{mei2023constrained}} and
  {\cite{wang2020mean}}, our condition is stricter due to $R / c$.
\end{remark}

The following remark is crucial, which shows that some specific umbilical caps
are static along the flow \eqref{flow}.

\begin{remark}
  \label{capstatic}For any $r > 0$, $q_{c, \theta}$ is identically zero on
  umbilical cap $\mathcal{C}_{c, r, \theta} (a)$. It is well-known that the
  mean curvature by conformality can be written as
  \begin{eqnarray*}
    H & = & e^{- w} (H_{\delta} + n \tilde{\nabla}_{\tilde{\nu}} w)\\
    & = & n x_{n + 1} \left( \frac{1}{r} + \tilde{\nabla}_{\tilde{\nu}} \log
    \frac{1}{x_{n + 1}} \right)\\
    & = & n x_{n + 1} \left( \frac{1}{r} - x^{- 1}_{n + 1} \langle \frac{x +
    r \cos \theta E_1 - c E_{n + 1} }{r}, E_{n + 1} \rangle_{\delta} \right) 
    \\
    & = & n \frac{c}{r},
  \end{eqnarray*}
  where $\tilde{\nu} = r^{- 1} (x + r \cos \theta E_1 - c E_{n + 1}) $ and we
  use the fact that the mean curvature of $\mathcal{C}_{c, r, \theta} (0)$
  with respect to the metric $\delta$ is $H_{\delta} = \frac{n}{r}$. Hence, we
  get
  \begin{eqnarray*}
    q_c& = & \frac{n c}{x_{n + 1}} - n \langle E_1, \nu \rangle c \cos \theta -
    H \langle x - c E_{n + 1}, \nu \rangle  \\
    &= & \tfrac{n c}{x_{n + 1}} - n c x_{n + 1}^{- 1} \bar{\nu}_1 \cos \theta -
    \tfrac{n c}{r} \langle - r \cos \theta E_1 + c E_{n + 1} + r \bar{\nu} - c
    E_{n + 1}, x_{n + 1} \bar{\nu} \rangle  \\
    &= & \tfrac{n}{x_{n + 1}} (- c \bar{\nu}_1 \cos \theta + c \tfrac{r \cos
    \theta}{r} \bar{\nu}_1)  \\
    &= & 0. 
  \end{eqnarray*}
  Therefore umbilical caps are static along the flow \eqref{flow}.
  
  Note that the principal curvature of an umbilical cap depends not only on
  the radius but also on the last coordinate of its center (in the Euclidean
  metric sense).
\end{remark}

The paper is organized as follows:

In Section \ref{prelim}, we introduce basic notations and definitions of
hypersurfaces. In Section \ref{sec:minkowski}, we prove a new Minkowski
formula on capillary hypersurface supported on a totally geodesic hyperplane,
comparing to the one in {\cite{chen2022some}}. In Sections \ref{sec:scalar},
\ref{sec:c0} and \ref{sec:c1}, we follow the method in {\cite{wang2020mean}}
and {\cite{mei2024constrained}} to study the scalar equation of the flow
\eqref{flow}, in particular, we prove the $C^0$ and $C^1$ estimates. In
Section \ref{sec:convergence}, we prove the uniform convergence of the flow.

\

\begin{acknowledgments*}
  We would like to thank Juncheol Pyo for helpful comments and hospitality. X.
  Chai has been partially supported by National Research Foundation of Korea
  grant No. 2022R1C1C1013511. Y. Chen has been supported by National Research
  Foundation of Korea grant No. RS-2023-00247299 and partially supported by
  National Research Foundation of Korea grant No. NRF-2020R1A01005698.
\end{acknowledgments*}

\section{Preliminaries}\label{prelim}

Throughout this paper, we assume that $x : M \rightarrow P_+ \subset
\mathbb{H}^{n + 1} $ is an embedding of capillary hypersurface along $P$. We
denote the second fundamental form of the embedding by $h_{i j}$. Since $\nu$
is the outer normal field on $\Sigma$, $h_{i j} = \langle \nabla_{e_i} \nu,
e_j \rangle$. Let $\kappa = (\kappa_1, \cdots, \kappa_n)$ be the eigenvalues
of $(h_{i j})$, i.e., the principal curvatures of $\Sigma$. The $k$-th mean
curvature $S_k$ is defined by
\[ S_r = \frac{1}{r!} \sum_{1 \leq i_1 < \cdots < i_r \leq n} \kappa_1
   \kappa_2 \cdots \kappa_{i_r}, \]
and the normalized mean curvature is defined by $H_r = \binom{n}{r}^{- 1}
S_r$.

The following so-called Newton transformation defined on the tangent bundle
is essential to our formula:
\[ T_0 = \tmop{id}, \]
\[ T_k = S_k I - T_{k - 1} \circ h. \]
The following properties are well-known,
\[ \tmop{tr} (T_r) = (n - r) S_r = (n - r) \binom{n}{r} H_r, \]
\[ \tmop{tr} (T_r \circ h) = (r + 1) S_{r + 1} = (n - k) \binom{n}{r} H_{r +
   1} . \]
\qquad Let $\bar{N}$ denote the outer normal field of $P$ ($P_+$ is the
interior side), $\nu$ denote the outer normal field of the immersed
hypersurface $\Sigma$, $\bar{\nu}$ denote the outer normal field of $\partial
\Sigma \subset P$ and $\eta$ denote the outer conormal field along $\partial
\Sigma \subset \Sigma$. Without loss of generality, assuming $\theta$ as the
angle between $- \nu$ and $\bar{N}$, we have the following relation
\begin{equation}
  \left\{ \begin{array}{l}
    \mu = \sin \theta \overline{N} + \cos \theta \overline{\nu}\\
    \nu = - \cos \theta \overline{N} + \sin \theta \overline{\nu}
  \end{array} \right. . \label{angle}
\end{equation}
Then the following lemma is widely recognized, and we refer to
{\cite{WangXia2019Stablecapillary}} for its proof.

\begin{lemma}
  Let $x : \Sigma \rightarrow P_+ \subset \mathbb{H}^{n + 1}$ be an isometric
  immersion of a capillary hypersurface supported on $P$. Then $\mu$ is a
  principal direction of $\Sigma$, that is,
  \begin{equation}
    \bar{\nabla}_{\mu} \nu = h (\mu, \mu) \mu . \label{pridir}
  \end{equation}
\end{lemma}

This property of capillary hypersurfaces is essential in the proof of the
Minkowski type formula in the next section.

\

\section{Minkowski type formula}\label{sec:minkowski}

In this section, we introduce a new Minkowski type formula, which is based on
the the following properties (see {\cite{guo2022stable}}).

\begin{proposition}
  \label{grad}It satisfies that
  \begin{equation}
    \bar{\nabla}_Z x = - \langle Z, \overline{E}_{n + 1} \rangle x + \langle
    Z, x \rangle \overline{E}_{n + 1}, \label{gradx}
  \end{equation}
  \begin{equation}
    \bar{\nabla}_Z  \bar{x} = \langle \bar{x}, Z \rangle \bar{E}_{n + 1},
    \label{gradxbar}
  \end{equation}
  \begin{equation}
    \bar{\nabla}_Z E_1 = - \langle Z, \overline{E}_{n + 1} \rangle E_1 +
    \langle Z, E_1 \rangle \overline{E}_{n + 1}, \label{gradE1}
  \end{equation}
  \begin{equation}
    \overline{\nabla }_Z (- E_{n + 1}) = \frac{1}{x_{n + 1}} Z, \label{gradE}
  \end{equation}
  and
  \begin{equation}
    \bar{\nabla}_Z (- \overline{E}_{n + 1}) = Z - \langle \overline{E }_{n +
    1}, Z \rangle \overline{E }_{n + 1} . \label{gradEbar}
  \end{equation}
\end{proposition}

These can be directly calculated by the relation between Levi-Civita
connections of the metrics conformal to each other, we refer the proof to
{\cite{guo2022stable}}. From \eqref{gradx}, \eqref{gradE1}, and \eqref{gradE},
we can easily see that $x$, $E_1$, and $E_{n + 1}$ are conformal Killing
vector fields, that is,
\begin{equation}
  (\mathcal{L}_x \bar{g}) (X, Y) = \frac{1}{2} (\langle \overline{\nabla }_Y
  x, X \rangle + \langle \overline{\nabla }_X x, Y \rangle) = 0, \label{confx}
\end{equation}
\begin{equation}
  (\mathcal{L}_{E_1} \bar{g}) (X, Y) = \frac{1}{2} (\langle \overline{\nabla
  }_Y E_1, X \rangle + \langle \overline{\nabla }_X E_1, Y \rangle) = 0,
  \label{confE1}
\end{equation}
and
\begin{equation}
  \begin{array}{ccl}
    \mathcal{L}_{- E_{n + 1}} \langle X, Y \rangle & = & \frac{1}{2} (\langle
    \overline{\nabla }_Y (- E_{n + 1}), X \rangle + \langle \overline{\nabla
    }_X (- E_{n + 1}), Y \rangle)\\
    & = & \frac{1}{x_{n + 1}} \langle X, Y \rangle, \label{confE}
  \end{array}
\end{equation}

for any $X, Y \in T\mathbb{H}^{n + 1}$, where $\bar{g}$ denotes the metric
$\langle \cdummy, \cdummy \rangle$.

Restricting the equation \eqref{confE} to $T \Sigma$, we have
\begin{equation}
  \langle \nabla^{\Sigma}_{e_i} E_1, e_j \rangle = - \langle e_i,
  \overline{E}_{n + 1} \rangle \langle E_1, e_j \rangle + \langle e_i, E_1
  \rangle \langle \overline{E}_{n + 1}, e_j \rangle - h_{i j} \langle E_1
  \comma \nu \rangle, \label{divE1}
\end{equation}
\begin{equation}
  \langle \nabla^{\Sigma}_{e_i} x, e_j \rangle = - \langle e_i,
  \overline{E}_{n + 1} \rangle \langle x, e_j \rangle + \langle e_i, x \rangle
  \langle \overline{E}_{n + 1}, e_j \rangle - h_{i j} \langle x \comma \nu
  \rangle, \label{divx}
\end{equation}
and
\begin{equation}
  \langle \nabla^{\Sigma}_{e_i} (- E_{n + 1}), e_j \rangle = \frac{1}{x_{n +
  1}} \langle e_i, e_j \rangle - h_{i j} \langle - E_{n + 1} \comma \nu
  \rangle . \label{divE}
\end{equation}
Let $x : \Sigma \rightarrow (\mathbb{H}^{n + 1}, \langle \cdummy, \cdummy
\rangle) = \left( \mathbb{R}^{n + 1}_+, \frac{1}{x^2_{n + 1}} \langle \cdummy
\nocomma, \cdummy \rangle_{\delta} \right)$ be an immersion of
$\theta$-capillary hypersurface supported on the hyperplane $P$. By using the
facts above, we can prove the following Minkowski type formula on $\Sigma$.

\begin{proposition}
  For $k = 1, \cdots, n,$and $c > 0$, it satisfies that
  \begin{equation}
    \int_{\Sigma} \left[ n c H_{k - 1} \left( \frac{1}{x_{n + 1}} - \cos
    \theta \langle E_1, \nu \rangle \right) - H_k \langle x - c E_{n + 1}, \nu
    \rangle \right] d A = 0, \label{Mink1}
  \end{equation}
  where $H_k$ is the $k$-th mean curvature of $\Sigma$.
\end{proposition}

\begin{proof}
  Let $T_r$ acts on the both side of $\eqref{divx} + c \eqref{divE}$ and
  integrate. Using divergence theorem, we have
  \begin{eqnarray*}
  	   &    & (n - r + 1) \binom{n}{r - 1} \int_{\Sigma} \left[ H_{r - 1} \left(
  	\frac{c}{x_{n + 1}} \right) - H_r \langle x - c E_{n + 1}, \nu \rangle
  	\right] d A\\
  	&= & \int_{\Sigma} \tmop{div}_{\Sigma} (T_{r - 1} (x - c E_{n + 1})) d
  	A\\
  	&= & \int_{\partial \Sigma} T_{r - 1} (x - c E_{n + 1}, \mu) d A\\
  	&= & \cos \theta \int_{\partial \Sigma} S_{r - 1 ; \mu} \langle x - c
  	E_{n + 1}, \bar{\nu} \rangle d s,
  \end{eqnarray*}

  where in the last equality we used \eqref{pridir}, the angle relation
  \eqref{angle} and the fact that $\bar{N} = - x_{n + 1} E_1$.
  
  Let $Z_{1, n + 1} = \langle E_1, \nu \rangle \bar{E}_{n + 1} - \langle
  \bar{E}_{n + 1}, \nu \rangle E_1$, applying Proposition \ref{grad}, we have
  \begin{eqnarray*}
    \langle \nabla_{e_i} Z_{1, n + 1}, e_j \rangle &=& - \delta_{i j}
    \langle E_1, \nu \rangle + h_{i k} [\langle E_1, e_k \rangle \langle
    \overline{E}_{n + 1}, e_j \rangle \\
    &&- \langle \overline{E}_{n + 1}, e_k
    \rangle \langle E_1, e_j \rangle] .
  \end{eqnarray*}
  Let $T_{r - 1}$ act on both sides, we get
  \[ \tmop{div}_{\Sigma}  (T_{r - 1} (Z_{1, n + 1})) = - (n - r + 1)
     \binom{n}{r - 1} H_{r - 1} \langle E_1, \nu \rangle . \]
  Considering the vector field $Z_1 = \overline{g} (E_1, \nu) \bar{x} -
  \overline{g} (\bar{x}, \nu) E_1$, from Proposition \ref{grad} we have
  \begin{eqnarray*}
    \langle \bar{\nabla}_{e_i} Z_1, e_j \rangle & = & \langle \bar{E}_{n + 1},
    \nu \rangle [\langle e_i, E_1 \rangle \langle \bar{x}, e_j \rangle -
    \langle e_i, \bar{x} \rangle \langle E_1, e_j \rangle]\\
    &  & + \langle E_1, \nu \rangle [\langle e_i, \bar{x} \rangle \langle
    \bar{E}_{n + 1}, e_j \rangle - \langle e_{i,} \bar{E}_{n + 1} \rangle
    \langle \bar{x}, e_j \rangle]\\
    &  & + h_{i k} [\langle E_1, e_k \rangle \langle \bar{x}, e_j \rangle -
    \langle \bar{x}, e_k \rangle \langle E_1, e_j \rangle]\\
    &  & + \langle \bar{x}, \nu \rangle [\langle \bar{E}_{n + 1}, e_i \rangle
    \langle E_1, e_j \rangle - \langle E_1, e_i \rangle \langle \bar{E}_{n +
    1}, e_j \rangle] .
  \end{eqnarray*}
  Similarly, we have
  \[ \tmop{div}_{\Sigma}  (T_{r - 1} (Z_1)) = 0. \]
  Combining all the equations above, we have

   \begin{eqnarray*}
       && (n - k + 1) \binom{n}{k - 1} \int_{\Sigma} \left(  \frac{c H_{k -
       1}}{x_{n + 1}} - H_k \langle x - c E_{n + 1}, \nu \rangle \right) d A\\
      & = & \cos \theta \int_{\partial \Sigma} S_{k - 1 ; \mu} \langle x - c
       E_{n + 1}, \bar{\nu} \rangle d s\\
      & = & \cos \theta \int_{\partial \Sigma} S_{k - 1 ; \mu} \langle Z_1 - c
       Z_{1, n + 1}, \mu \rangle d s\\
       &= & \cos \theta \int_{\partial \Sigma} T_{k - 1} (Z_1 - c Z_{1, n + 1},
       \mu) d s\\
      & = & \cos \theta \int_{\Sigma} \tmop{div}_{_{\Sigma}} (T_{k - 1} (Z_1 -
       c Z_{1, n - 1})) d s\\
      & = & (n - k + 1) \binom{n}{k - 1} \cos \theta \int_{\Sigma} H_{k - 1}
       \langle c E_1, \nu \rangle d A.
     \end{eqnarray*} 

  Therefore, we obtain the Minkowski type formula \eqref{Mink1}.
\end{proof}

Let $k = 1$, the Minkowski formula \eqref{Mink1} becomes
\begin{equation}
  \int_{\Sigma} \left[ n c \left( \frac{1}{x_{n + 1}} - \cos \theta \langle
  E_1, \nu \rangle \right) - H \langle x - c E_{n + 1}, \nu \rangle \right] d
  A = 0, \label{Mink}
\end{equation}
which holds for any capillary hypersurfaces $\Sigma$ supported on totally
geodesic hyperplane $P$. Here $H = n H_1$ is the mean curvature of $\Sigma$.

\begin{remark}
  The second author and Juncheol Pyo {\cite{chen2022some}} gave another version of Minkowski type formula on capillary hypersurfaces supported on a totally geodesic plane, which is presented in Poincar{\'e} ball model $(\mathbb{H}^{n + 1},
  \bar{g}) = (\mathbb{B}^{n + 1}, \frac{4}{(1 - | x |^2)^2} \delta)$ as follows ($\mathbb B^{n+1}$ is an Euclidean unit ball and $\delta$ is the Euclidean metric in $\mathbb B^{n+1}$).
  \begin{equation}
    \int_{\Sigma} [n V_0 - n \cos \theta \bar{g} (Y_{n + 1}, \nu) - H \bar{g}
    (x, \nu)] d A = 0, \label{MinkP}
  \end{equation}
  where $\Sigma$ is a $\theta$-capillary hypersurface supported on a totally
  geodesic hypersurface $P' = \{ x \in \mathbb{H}^{n + 1}, \bar{g} (x, E_{n +
  1}) = 0 \}$ (in Poincar{\'e} ball model), $V_0 = (1 + | x |^2) / (1 - | x
  |^2)$, $x$ is the position vector and $Y_{n + 1} = \delta (x, E_{n + 1}) x -
  \frac{1}{2} (1 + | x |^2) E_{n + 1}$. But unfortunately, we cannot find any
  umbilical $\theta$-capillary hypersurfaces where the integrand in
  \eqref{MinkP} is identically zero.
\end{remark}

Let $\hat{\Sigma}$ denote the domain enclosed by $\Sigma$ and $P$, and $|
\widehat{\partial \Sigma} |$ be the domain enclosed by $\partial \Sigma$ on
$P$. The energy functional defined by
\[ \mathcal{Q} (\Sigma) = | \Sigma | - \cos \theta | \widehat{\partial
   \Sigma} | \]
is well-known since the critical hypersurface of this functional under any
volume preserving variation is a $\theta$-capillary hypersurface with constant
mean curvature, see {\cite{ros1997stability}} and
{\cite{WangXia2019Stablecapillary}}. Under a flow $\Sigma_t = x_t (M)$ with
with the given normal velocity $f$ and capillary boundary condition as in
\eqref{flow}, the following variation formula is well-known:
\[ \frac{d}{d t} | \hat{\Sigma}_t | = \int_{\Sigma_t} f d A_t \]
and
\[ \frac{d}{d t} \mathcal{Q} (\Sigma_t) = \text{$\int_{\Sigma_t} H f d A_t
   .$} \]
From the Minkowski formula $\eqref{Mink}$, it is evident that the flow
described in \eqref{flow} is a volume preserving flow.

\section{Scalar equation of the flow}\label{sec:scalar}

In this section, we express the flow \eqref{flow} by a scalar equation of the
radius function.

Let $x$ be the position vector defined on $M$ which is represented by
\[ x = c E_{n + 1} + \rho (z) z, z \in \Omega \subset \bar{\mathbb{S}}^n_+, \]
where $\rho$ defined on $\mathbb{S}^n_+$ is the distance between $x$ and $c
E_{n + 1}$ in the Euclidean metric. Since $\Sigma$ is smooth, star-shaped, the
function $\rho$ is well-defined and smooth on $\Omega$. Let $u = \log \rho$,
then $u \in C^2 (\Omega) \cap C^0 (\bar{\Omega})$. Define $v = \sqrt{1 + |
\nabla u |^2}$, where $\nabla u$ is the gradient of $u$ with respect to the
ordinary metric on $\mathbb{S}^n_+$. From the basic facts for radial function,
it is well-known that
\[ \tilde{\nu} : = x^{- 1}_{n + 1} \nu = \frac{\zeta - \rho^{- 1} \nabla
   u}{v}, \]
where $\zeta = \partial_{\rho}$.

Throughout the paper, we let the indices $i, j, k$ range from 1 to $n$ and we
will apply the Einstein convention.

We use polar coordinates $(\rho, \beta, \gamma, \xi) \in [0, + \infty) \times
\left[ 0, \frac{\pi}{2} \right] \times [0, 2 \pi] \times \mathbb{S}^{n - 2} $,
where $\xi$ is the spherical coordinate on $\mathbb{S}^{n - 2}$, and the
star-shaped hypersurface $\Sigma \assign x (M)$ can be written as
\begin{equation}
  x - c E_{n + 1} = \rho (z) z = \rho (\beta, \gamma, \xi) z, \hspace{1.8em}
  \tmop{where} z \assign (\beta, \gamma, \xi) \in \bar{\mathbb{S}}_+^n,
  \label{sphcoord}
\end{equation}
where $x_1 = \rho \cos \beta$ and $x_{n + 1} = \rho \sin \beta \cos \gamma + c
= e^{- w}$. Then the Euclidean metric $d s^2 = \langle \cdummy, \cdummy
\rangle_{\delta}$ can be written as
\[ d s^2 = d \rho^2 + \rho^2 \sigma_{\mathbb{S}^n_+} = d \rho^2 + \rho^2 d
   \beta^2 + \rho^2 \sin^2 \beta d \gamma^2 + \rho^2 \sin^2 \beta \sin^2
   \gamma \sigma_{\mathbb{S}^{n - 2} }, \]
and we have
\[ \rho^2 = x_1^2 + \sum_{i = 2}^n x_i^2 + (x_{n + 1} - c)^2 . \]
Now we can represent $E_1$, $E_{n + 1}$ by the coordinate \eqref{sphcoord}.
Indeed, since
\[ \frac{\partial \rho}{\partial x_1} = \frac{x_1}{\rho} = \cos \beta, \]
and
\[ - \sin \beta \frac{\partial \beta}{\partial x_1} = \frac{\partial (\cos
   \beta)}{\partial x_1} = \frac{\partial (x_1 / \rho)}{\partial x_1} =
   \frac{\sin^2 \beta}{\rho}, \]
we can represent $E_1$ by the polar coordinate defined above,
\begin{equation}
  E_1 = \frac{\partial}{\partial x_1} = \frac{\partial \rho}{\partial x_1}
  \partial_{\rho} + \frac{\partial \beta}{\partial x_1} \partial_{\beta} =
  \cos \beta \partial_{\rho} - \frac{\sin \beta}{\rho} \partial_{\beta} .
  \label{repE1}
\end{equation}
Similarly, since
\[ \frac{\partial \rho}{\partial x_{n + 1}} = \frac{x_{n + 1} - c}{\rho} =
   \sin \beta \cos \gamma, \]
\[ - \sin \beta \frac{\partial \beta}{\partial x_{n + 1}} = \frac{\partial
   (\cos \beta)}{\partial x_{n + 1}} = \frac{\partial (x_1 / \rho)}{\partial
   x_{n + 1}} = - \frac{1}{\rho} \sin \beta \cos \beta \cos \gamma \]
and
\[ - \sin \gamma \frac{\partial \gamma}{\partial x_{n + 1}} = \frac{\partial
   (\cos \gamma)}{\partial x_{n + 1}} = \frac{\partial ((x_{n + 1} - c) /
   (\rho \sin \beta))}{\partial x_{n + 1}} = \frac{\sin^2 \gamma}{\rho \sin
   \beta}, \]
we have
\begin{equation}
  E_{n + 1} = \frac{\partial}{\partial x_{n + 1}} = \sin \beta \cos \gamma
  \partial_{\rho} + \frac{1}{\rho} \cos \beta \cos \gamma \partial_{\beta} -
  \frac{\sin \gamma}{\rho \sin \beta} \partial_{\gamma} . \label{repE}
\end{equation}
Denoting $u_{\beta} = \sigma (\nabla u, \partial_{\beta})$ and $u_{\gamma} =
\sigma (\nabla u, \partial_{\gamma})$, from \eqref{repE1} and \eqref{repE} we
have
\begin{equation}
  \langle E_1, \nu \rangle = e^{2 w} \langle \cos \beta \partial_{\rho} -
  \frac{\sin \beta}{\rho} \partial_{\beta}, e^{- w} \frac{\zeta - \rho^{- 1}
  \nabla u}{v} \rangle_{\delta} = e^w \frac{\cos \beta - \sin \beta
  u_{\beta}}{v} \label{suppE1},
\end{equation}
and
\begin{equation}
  \begin{aligned}
    \langle E_{n + 1}, \tilde{\nu} \rangle_{\delta} = & \langle \sin \beta
    \cos \gamma \partial_{\rho} + \frac{1}{\rho} \cos \beta \cos \gamma
    \partial_{\beta} - \frac{\sin \gamma}{\rho \sin \beta} \partial_{\gamma},
    \frac{\zeta - \rho^{- 1} \nabla u}{v} \rangle_{\delta}\\
    = & \frac{1}{v} (\sin \beta \cos \gamma - \cos \beta \cos \gamma
    u_{\beta} + \sin \beta \sin \gamma u_{\gamma}) .
  \end{aligned} \label{suppE}
\end{equation}
By definition, we have
\begin{equation}
  \langle x - c E_{n + 1}, \nu \rangle = e^{2 w} \langle \rho \zeta, e^{- w}
  \frac{\zeta - \rho^{- 1} \nabla u}{v} \rangle = \frac{\rho e^w}{v} .
  \label{supp}
\end{equation}
Moreover, by the conformality of mean curvature,
\begin{equation}
  \begin{aligned}
    H  = & e^{- w} \left[ \frac{n}{\rho v} - \frac{1}{\rho v} \left(
    \sigma^{i j} - \frac{u^i u^j}{v^2} \right) u_{i j} \right] - n
    D_{\tilde{\nu}} e^{- w}\\
    = & e^{- w} \left[ \frac{n}{\rho v} - \frac{1}{\rho v} \left( \sigma^{i
    j} - \frac{u^i u^j}{v^2} \right) u_{i j} \right] - n \langle E_{n + 1},
    \tilde{\nu} \rangle_{\delta}\\
    = & - \frac{1}{\rho v e^w} \left( \sigma^{i j} - \frac{u^i u^j}{v^2}
    \right) u_{i j} + \frac{1}{v} (\cos \beta \cos \gamma u_{\beta} - \sin
    \beta \sin \gamma u_{\gamma}) + \frac{n c}{\rho v},
  \end{aligned} \label{confH}
\end{equation}
where $\sigma^{i j}$ corresponds to the inverse of the metric on
$\mathbb{S}^n_+$, $u^i = \sigma^{i j} u_j$ is the $i$-th component of $\nabla
u$ (the gradient of $u$ on $\mathbb{S}^n_+$) and $u_{i j}$ is the $(i, j)$-th
component of the Hessian $\nabla^2 u$ on $\mathbb{S}^n_+$. Then by
calculation,
\begin{eqnarray*}
     q_c & = & n c \left( \frac{1}{x_{n + 1}} - \cos \theta \langle E_1, \nu
     \rangle \right) - H \langle x - c E_{n + 1}, \nu \rangle\\
     & = & n c e^w - n c \cos \theta v^{- 1} e^w (\cos \theta - \sin \beta
     u_{\beta})\\
     &  & - \left( - \frac{1}{\rho v e^w} \left( \sigma^{i j} - \frac{u^i
     u^j}{v^2} \right) u_{i j} + \frac{1}{v} (\cos \beta \cos \gamma u_{\beta}
     - \sin \beta \sin \gamma u_{\gamma}) + \frac{n c}{\rho v} \right)
     \frac{\rho e^w}{v} .
\end{eqnarray*}
Writing the flow as a function $u$ of $t$, from the flow \eqref{flow}, we see
that
\[ q_c = \langle \partial_t x, \nu \rangle = \rho_t e^{- w} e^{2 w} \langle
   \zeta, \frac{\zeta - \rho^{- 1} \nabla u}{\sqrt{1 + | \nabla u |^2}}
   \rangle_{\delta} = \frac{\rho_t e^w}{\sqrt{1 + | \nabla u |^2}} = \frac{u_t
   \rho e^w}{v} . \]
Therefore, we obtain the evolution equation of $u$ as follows,
\begin{equation}
  \begin{aligned}
    u_t = & \frac{v}{\rho e^w} q_c\\
     =& \frac{n c v}{\rho} - \frac{n c \cos \theta}{\rho} (\cos \beta -
    \sin \beta u_{\beta}) + \frac{1}{\rho e^w v} \left( \sigma^{i j} -
    \frac{u^i u^j}{v^2} \right) u_{i j}\\
    &   + \frac{n}{v} (- \cos \beta \cos \gamma u_{\beta} + \sin \beta \sin
    \gamma u_{\gamma}) - \frac{n c}{\rho v}\\
     = & \frac{n c}{\rho} \frac{| \nabla u |^2}{v} - \frac{n c \cos
    \theta}{\rho} (\cos \beta - \sin \beta u_{\beta})\\
      & + \frac{1}{\rho e^w v} \left( \sigma^{i j} - \frac{u^i u^j}{v^2}
    \right) u_{i j} + \frac{n}{v} (- \cos \beta \cos \gamma u_{\beta} + \sin
    \beta \sin \gamma u_{\gamma})\\
     \backassign & Q_c (\nabla^2 u, \nabla u, \rho, \beta, \gamma) .
  \end{aligned} \label{evolscal}
\end{equation}
As for the boundary condition in \eqref{flow},
\begin{equation*}
  \begin{aligned}
    - \cos \theta & = \langle \nu, \bar{N} \rangle\\
    & = e^{2 w} \langle e^{- w} \tilde{\nu}, - e^{- w} E_1 \rangle_{\delta}\\
    & = \langle \frac{\zeta - \rho^{- 1} \nabla u}{v}, \frac{1}{\rho} \partial_{\beta} \rangle_{\delta}\\
    & = - \frac{1}{v} u_{\beta} .
  \end{aligned}
\end{equation*}
Therefore we have
\begin{equation}
  u_{\beta} = \rho v \cos \theta = \sqrt{1 + | \nabla u |^2} \cos \theta .
  \label{bdryscal}
\end{equation}
Combining \eqref{evolscal} and \eqref{bdryscal}, we know that the flow
\eqref{flow} can be written as the parabolic equation of the scalar function
$u$ as follows,
\begin{equation}
  \left\{\begin{array}{lllll}
    u_t = Q_c (\nabla^2 u, \nabla u, \rho, \beta, \gamma), &  &  & \tmop{in} &
    \mathbb{S}^n_+ \times [0, T) ;\\
    u_{\beta} = \sqrt{1 + | \nabla u |^2} \cos \theta, &  &  & \tmop{on} &
    \partial \mathbb{S}^n_+ \times [0, T) ;\\
    u (\cdummy, 0) = u_0 (.), &  &  & \tmop{on} & \mathbb{S}^n_+,
  \end{array}\right. \label{scalflow}
\end{equation}
where $u_0 = \log \rho_0$ and $\rho_0 = | x_0 - c E_{n + 1} |_{\delta}$ is the
radial function with respect to $c E_{n + 1}$ of the initial hypersurface
$\Sigma_0 = x_0 (M)$.

The short time existence of the flow can be guaranteed by applying the
standard PDE theory to \eqref{scalflow}, due to the assumption on the
star-shapedness of $\Sigma_0 = x_0 (M)$. In the following section, we will
show the uniform $C^0$ and $C^1$-estimates for the equation. \

Before we start, we need the following calculation.
\begin{equation}
  Q_c^{i j} \assign \left. \frac{\partial Q_c (\lambda, \psi, \rho, \beta,
  \gamma)}{\partial \lambda_{i j} } \right|_{\lambda = \nabla^2 u, \psi =
  \nabla u} = \frac{1}{\rho v e^w} \left( \sigma^{i j} - \frac{u^i u^j}{v^2}
  \right), \label{Fij}
\end{equation}
\begin{equation}
  \begin{aligned}
    Q_{c, \psi_i} & \assign \left. \frac{\partial Q_c (\lambda, \psi, \rho, \beta, \gamma)}{\partial \psi_i} \right|_{\lambda = \nabla^2 u, \psi = \nabla u}\\
    & = \frac{n c}{\rho} \left( \frac{2 u_i}{v} - \frac{| \nabla u |^2 u_i}{v^3} \right) - \frac{u_i}{\rho e^w v^3} a^{k l} u_{k l} - \frac{2}{\rho e^w v^3} a^{i l} u_{k l} u_k\\
    & \quad + \frac{n c \cos \theta \sin \beta}{\rho} \sigma (\partial_{\beta}, e_i) + \frac{n u_i}{v^3} (- \cos \beta \cos \gamma u_{\beta} + \sin \beta \sin \gamma u_{\gamma})\\
    & \quad - \frac{n}{v} (\cos \beta \cos \gamma \sigma (\partial_{\beta}, e_i) - \sin \beta \sin \gamma \sigma (\partial_{\gamma}, e_i)),
  \end{aligned} 
\end{equation}
\begin{equation}
  \begin{aligned}
    Q_{c, \rho} & \assign \left. \frac{\partial Q_c (\lambda, \psi, \rho, \beta, \gamma)}{\partial \rho} \right|_{\lambda = \nabla^2 u, \psi = \nabla u}\\
    & = \frac{n c}{\rho^2} \frac{| \nabla u |^2}{v} + \frac{n c \cos \theta}{\rho^2} (\cos \beta - \sin \beta u_{\beta}) - \frac{1}{\rho^2 v} a^{i j} u_{i j},
  \end{aligned} 
  \label{Frho}
\end{equation}
\begin{equation}
  \begin{aligned}
    Q_{c, \beta} & \assign \left. \frac{\partial Q_c (\lambda, \psi, \rho, \beta, \gamma)}{\partial \beta} \right|_{\lambda = \nabla^2 u, \psi = \nabla u}\\
    & = \frac{1}{v} \cos \beta \cos \gamma a^{i j} u_{i j} + n c \cos \theta \left( \frac{\sin \beta}{\rho} + \frac{\cos \beta}{\rho } u_{\beta} \right)\\
    & \quad + \frac{n}{v} (\sin \beta \cos \gamma u_{\beta} + \cos \beta \sin \gamma u_{\gamma}),
  \end{aligned} 
  \label{Fbeta}
\end{equation}
\begin{equation}
  \begin{aligned}
    Q_{c, \gamma} & \assign \left. \frac{\partial Q_c (\lambda, \psi, \rho, \beta, \gamma)}{\partial \gamma} \right|_{\lambda = \nabla^2 u, \psi = \nabla u}\\
    & = - \frac{1}{v} \sin \beta \sin \gamma a^{i j} u_{i j} + \frac{n}{v} (\cos \beta \sin \gamma u_{\beta} + \sin \beta \cos \gamma u_{\gamma}),
  \end{aligned} 
  \label{Fgama}
\end{equation}

where $a^{i j} = \sigma^{i j} - \frac{u^i u^j}{v^2}$. Now we are ready to
prove the $C^0$ and $C^1$-estimates.

\section{$C^0$ estimate}\label{sec:c0}

Wang and Weng {\cite{wang2020mean}} proved a $C^0$-estimate for a similar
constrained mean curvature type flow of capillary hypersurfaces in the
Euclidean unit ball. More specifically, if the initial surface is bounded by
two spherical caps that remains stationary under a specified constrained mean
curvature flow, the flow will remain bounded by the same two spherical caps.
Therefore, a certain pair of spherical caps can be used as {\tmem{barriers}}
of their designed constrained mean curvature flow. A similar property holds in
the case of geodesic ball in space forms, see {\cite{mei2023constrained}}.

According to the discussion in Remark \ref{capstatic}, the umbilical caps
defined by \eqref{capbarrier} can be regarded as barriers of the flow
\eqref{flow}. Then we have a similar corresponding $C^0$-estimate.

\begin{proposition}
  \label{C0}Let $x_0 (M)$ be an initial star-shaped hypersurface with respect
  to $c E_{n + 1}$ which satisfies that
  \[ x_0 (M) \subset \mathcal{C}_{c, R_1, \theta} \backslash\mathcal{C}_{c,
     R_2, \theta} \]
  for some $R_1 > R_2 > 0$, where $\mathcal{C}_{c, R, \theta}$ is defined in
  \eqref{capbarrier}. Then it holds that
  \[ x (\nobracket \nobracket M, t) \subset \mathcal{C}_{c, R_1, \theta}
     \backslash\mathcal{C}_{c, R_2, \theta} \]
  \begin{flushleft}
    for all $t$ along the flow \eqref{flow}. In particular, if $u (x, t)$
    solves the initial boundary value problem \eqref{scalflow} on $[0,\infty)$, then for any
    $T > 0$,
  \end{flushleft}
  \[ \| u \|_{C^0 (\mathbb{S}^{n }_+ \times [0, T])} \leq C, \]
  where $C = C (u_0, \nabla u_0, \nabla^2 u_0)$.
\end{proposition}

\begin{proof}
  Let $\varphi$ be the defining logarithmic radial function of $C_{c, R_1,
  \theta}$. Then, $\varphi$ is a static solution of the equation of
  \eqref{scalflow}, we have
  \begin{eqnarray*}
    \partial_t (u - \varphi) & = & Q_c (\nabla^2 u, \nabla u, e^u, \beta,
    \gamma) - Q_c (\nabla^2 \varphi, \nabla \varphi, e^{\varphi}, \beta,
    \gamma)\\
    & = & A^{i j} \nabla_{i j} (u - \varphi) + b^j \cdot (u - \varphi)_j + B
    (u - \varphi),
  \end{eqnarray*}
  where
  \[ A^{i j} = \int_0^1 Q_c^{i j} (\nabla^2 (s u + (1 - s) \varphi), \nabla (s
     u + (1 - s) \varphi), e^{s u + (1 - s) \varphi}, \beta, \gamma) d s, \]
  
  \[ b_j = \int_0^1 Q_{c, p_j} (\nabla^2 (s u + (1 - s) \varphi), \nabla (s u
     + (1 - s) \varphi), e^{s u + (1 - s) \psi}, \beta, \gamma) d s \]
  and
  \[ B = \int_0^1 Q_{c, \rho} (\nabla^2 (s u + (1 - s) \varphi), \nabla (s u +
     (1 - s) \varphi), e^{s u + (1 - s) \psi}, \beta, \gamma) e^{s u + (1 - s)
     \psi} d s. \]
  Since $Q_{c, \rho}$ has no singular point if $\rho \neq 0$, $B$ is bounded
  and we can denote $\lambda = - \sup_{\mathbb{S}^n_+ \times [0, T]} | B |$.
  We get
  \[ \partial_t (e^{\lambda t} (u - \varphi)) = e^{\lambda t} (A^{i j}
     \nabla_{i j} (u - \varphi) + b^j \cdot (u - \varphi)_j + (B + \lambda) (u
     - \varphi)) . \]
  Let $(y_0, t_0)$ be the point where $e^{\lambda t} (u - \varphi)$ attains \
  its nonnegative maximum value. By maximum principle, $(y_0, t_0)$ can only
  be located on the parabolic boundary, say $(y_0, t_0)$. That is,
  \[ e^{\lambda t} (u (x, t) - \varphi (x)) \leq \sup_{\partial \mathbb{S}^n_+
     \times [0, T) \cup \mathbb{S}^n_+ \times \{ 0 \}} \{ 0, e^{\lambda t} (u
     (x, t) - \varphi (x)) \} . \]
  \ If $y_0 \in \partial \mathbb{S}^n_+,$ by Hopf Lemma we have
  \[ \nabla^{ \partial \mathbb{S}^n_+} (u - \varphi) (y_0, t_0) = 0 ;
     \quad \nabla_n (u - \varphi) (y_0, t_0) < 0, \]
  where $\nabla^{ \partial \mathbb{S}^n_+}$ denotes the gradient of a
  function on $\partial \mathbb{S}^n_+$, and $\nabla_n$ is the normal
  derivative on $\partial \mathbb{S}^n_+$. Then we have
  \[ | \nabla^{\partial \mathbb{S}^n_+} u (y_0, t_0) | = |
     \nabla^{\partial \mathbb{S}^n_+} \varphi (y_0, t_0) | \assign k,
     \infixand \nabla_n u (y_0, t_0) < \nabla_n \varphi (y_0, t_0) . \]
  \begin{flushleft}
    From the boundary condition in \eqref{scalflow},
    \[ \frac{\nabla_n u}{\sqrt{1 + k^2 + | \nabla_n u |^2}} = - \cos \theta =
       \frac{\nabla_n \varphi}{\sqrt{1 + k^2 + | \nabla_n \varphi |^2}}, \]
    which is a contradiction to $\nabla_n u (y_0, t_0) < \nabla_n \varphi
    (y_0, t_0)$ by monotonicity of the function $g (x) = x / \sqrt{1 + k^2 +
    x^2} .$
  \end{flushleft}
  
  Therefore we can only have $t_0 = 0$, that is,
  \[ e^{\lambda t} (u (y, t) - \varphi (y, t)) \leq u_0 (y_0) - \varphi (y_0)
     \leq 0, \quad \forall \; (y, t) \in \mathbb{S}^n_+ \times [0, T], \]
  which gives $u$ an upper bound by the umbilical cap $\mathcal{C}_{c, R_2,
  \theta}$'s defining function $\varphi$. Similarly, the desired lower bound
  can be obtained. We have finished the proof of Proposition \ref{C0}. 
\end{proof}

\section{$C^1$-estimate}\label{sec:c1}

In this section, we prove a $C^1$-estimate of the flow \eqref{scalflow}.
Inspired by {\cite{wang2020mean}}, we use the similar auxiliary function. Let
$d : \tmop{neigh} (\partial \mathbb{S}^n_+) \rightarrow \mathbb{R}$ be a
nonnegative smooth function on $\tmop{neigh} (\partial \mathbb{S}^n_+) \subset
\mathbb{S}^n_+$ defined by
\[ d (x) \assign \tmop{dist}_{\sigma} (x, \partial \mathbb{S}^n_+) . \]
Indeed, the function is well defined on a neighborhood of $\partial
\mathbb{S}^n_+$ but it can be extended to the whole $\mathbb{S}^n_+$ and
satisfies that
\[ d \geq 0, \hspace{1.2em} | \nabla d | \leq 1, \; \tmop{in} \hspace{1.2em}
   \bar{\mathbb{S}}^n_+ . \hspace{1.2em} \]
The following $C^1$-estimate is essential for the flow \eqref{flow}.

\begin{proposition}
  If $\theta$ satisfies that $| \cos \theta | < \frac{4 n K_0 (c, R, \theta) -
  c (n - 1)}{4 n K_0 (c, R, \theta) + c (n - 1)}$, for any $(x, t) \in
  \mathbb{S}^n_+ \times [0, T]$ $(T < T^{\ast})$, we have
  \[ | \nabla u | (x, t) \leq C, \]
  for some positive constant $C = C (\nabla^2 u_0, \nabla u_0, u_0)$.
\end{proposition}

\begin{proof}
  Define an auxiliary function inspired by {\cite{wang2020mean}},
  \[ \Psi \assign (1 + K d) v + \cos \theta \sigma (\nabla u, \nabla d) \]
  where $K$ is a constant to be determined later. Let $(y_0, t_0) \in
  \bar{\mathbb{S}}^n_+ \times [0, t]$ be the point where $\Psi$ attains its
  maximum. We will discuss case by case to prove the theorem.
  
  {\tmstrong{Case 1}}: $(y_0, t_0) \in \partial \mathbb{S}^n_+ \times [0, T]$.
  At the point $x_0 \in \partial \mathbb{S}^n_+$, we choose a local coordinate
  $\{ y_1, \cdots, y_n \}$ near $y_0$, such that $\frac{\partial}{\partial
  y_1} = - \partial_{\beta}$ is an inner normal vector of $\partial
  \mathbb{S}^n_+$, and $\{ y_i \}_{i = 2}^{n + 1}$ be the geodesic coordinate
  near $y_0 \in \partial \mathbb{S}^n_+$ along $y_1 = t$, $(0 < t <
  \varepsilon)$ in the neighborhood of $x_0$. Under this coordinate,
  $\frac{\partial}{\partial y_1} = \nabla d$ on $\partial \mathbb{S}^n_+$.
  
  First of all, from the boundary condition of \eqref{scalflow}, we know that
  $u_1  \assign \nabla_{- \partial_{\beta}} u = - \cos \theta v$, denote
  $\nabla' u = \nabla u - u_1$, then from $1 + | \nabla' u |^2 = v^2 - u_1^2 =
  \sin \theta v^2$ we have,
  \[ \nobracket \Psi |_{\partial \mathbb{S}^n_+} = v - v \cos^2 \theta =
     \sqrt{1 + | \nabla' u |^2 + u_1^2} \sin^2 \theta = | \sin \theta |
     \sqrt{1 + | \nabla' u |^2} . \]
  Applying the Gauss-Weingarten equation, we have
 \begin{equation*}
  \begin{aligned}
    \nabla_1 v & = \frac{\nabla_k u \nabla_{k 1} u}{v} = \frac{\sum_{i = 2}^n \nabla_i u \nabla_{1 i} u}{v} - \cos \theta \nabla_{11} u\\
    & = \frac{1}{v} \sum_{i = 2}^n \left( u_i u_{i 1} + \sum_{j = 2}^n u_i h^{\partial \mathbb{S}^n_+}_{i j} u_j \right) - \cos \theta \nabla_{11} u\\
    & = \sum_{i = 2}^n \frac{u_i u_{i 1}}{v} - \cos \theta \nabla_{11} u,
  \end{aligned} 
\end{equation*}
  where we have used the fact that $h^{\partial \mathbb{S}^n_+}_{i j} = 0$
  since $\partial \mathbb{S}^n_+$ is totally geodesic in $\mathbb{S}^n_+$.
  Then, from Hopf Lemma we have,
\begin{equation}
  \begin{aligned}
    0 & \geq \nabla_1 \Psi (x_0) = \nabla_1 v + K v \nabla_1 d + \nabla_1 (u_k d_k) \cos \theta\\
    & = \nabla_1 v + K v \nabla_1 d + \nabla_{k 1} u d_k \cos \theta + u_k \nabla_{k 1} d \cos \theta\\
    & = \frac{1}{v} \sum_{i = 2}^n u_i u_{1 i} + K v + u_k \nabla_{k 1} d \cos \theta,
  \end{aligned} 
  \label{bdryderiY}
\end{equation}
  and for $i \geq 2$, we have
  \[ \nabla_i' \Psi (x_0) = v_i + u_{1 i} \cos \theta . \]
  Differentiating the boundary condition \eqref{scalflow}, together with the
  fact above we have
  \[ u_{1 i} = - \nabla_i' (\cos \theta v) = \cos^2 \theta u_{1 i}, \]
  therefore we know
  \begin{equation}
    u_{1 i} = 0, \quad \forall 1 \leq i \leq n - 1. \label{distinctu1i}
  \end{equation}
  Then applying \eqref{distinctu1i} to \eqref{bdryderiY}, we have
  \[ 0 \geq K v  + u_k \nabla_{k 1} d \cos \theta \geq \Psi \left(
     \frac{K}{\sin^2 \theta} - C_1 \right), \]
  where $C_1$ is a universal constant which satisfies $\frac{| | \nabla u |
  \cos \theta |}{v | \sin^2 \theta |} \leq C_1$. Then $K$ can be chosen large
  enough so that it comes to a contradiction.
  
  {\tmstrong{Case 2:}} If $(x_0, t_0) \in \bar{\mathbb{S}}^n_+ \times \{ 0
  \}$, then
  \[ \Psi (x, t) \leq \Psi (x_0, 0) = (1 + K d) \sqrt{1 + | \nabla u_0 |^2} +
     \cos \theta \sigma (\nabla  u_0, \nabla d) \leq C_2 . \]
  Therefore $v (x, t) \leq C,$ for any $(x, t) \in \bar{\mathbb{S}} \times [0,
  T]$.
  
  {\tmstrong{Case 3:}} In the case of $(x_0, t_0) \in \tmop{int}
  (\mathbb{S}^n_+)$, we have
\begin{equation}
  \begin{aligned}
    0  = &\nabla_i \Psi (x_0, t_0) = (1 + K d) v_i + K d_i v + \cos \theta (u_k d_k)_i\\
     = &\left( (1 + K d) \frac{\nabla_k u}{v} + \nabla_k d \cos \theta \right) \nabla_{i k} u + \nabla_k u \nabla_{i k} d \cos \theta + K d_i v.
  \end{aligned} 
  \label{intcriY}
\end{equation}
  Choose a geodesic coordinate, up to a rotation of the one in {\tmstrong{Case
  1}}, such that
  \[ u_1 = | \nabla u | > 0, \quad \tmop{and} \quad \{ \nabla_{\alpha \beta}
     u \}_{2 \leq \alpha, \beta \leq n} \, \tmop{is} \, \tmop{diagonal} . \]
  Assume $u_1 = | \nabla u |$ is large enough such that $u_1$, $v = \sqrt{1 +
  u_1^2}$ and $\Psi = (1 + K d) v + u_1 d_1 \cos \theta$ are equivalent.
  Otherwise, we could obtain the desire $C^1$ estimate for $u$. Here we denote
  $d_1 = \sigma (\nabla d, e_1) = u_1^{- 1} \sigma (\nabla d, \nabla u)$.
  
  Firstly, from \eqref{intcriY} by letting $i = \alpha$, we have
  \begin{equation}
    \left[ (1 + K d) \frac{u_1}{v} + \cos \theta d_1 \right] u_{1 \alpha} = -
    \cos \theta u_{\alpha \alpha} d_{\alpha} - \cos \theta u_1 d_{1 \alpha} -
    K d_{\alpha} v, \label{uij1}
  \end{equation}
  and letting $i = 1$, we have
  \begin{equation}
    \left[ (1 + K d) \frac{u_1}{v} + \cos \theta d_1 \right] u_{11} = - \cos
    \theta u_{\alpha 1} d_{\alpha} - \cos \theta u_1 d_{11} - K d_1 v,
    \label{uij2}
  \end{equation}
  We can see from the \eqref{uij1} that
  \begin{equation}
    u_{1 \alpha} = - S^{- 1} \cos \theta u_{\alpha \alpha} d_{\alpha} - S^{-
    1} (\cos \theta u_1 d_{1 \alpha} + K d_{\alpha} v), \label{uij3}
  \end{equation}
  where $S = (1 + K d) \frac{u_1}{v} + \cos \theta d_1$, and it is clear that
  $2 + \pi K \geq S \geq C (\delta, \theta)$, if there exists some positive
  $\delta > 0$ such that $u_1 (x_0, t_0) \geq \delta$. Indeed, if not, we
  would have already proved the theorem. \
  
  Inserting \eqref{uij3} into \eqref{uij2},
\begin{equation}
  \begin{aligned}
    u_{11}  = &- S^{- 1} \cos \theta u_{\alpha 1} d_{\alpha} - S ^{- 1} (\cos \theta u_1 d_{11} + K d_1 v)\\
     = &\frac{\cos^2 \theta}{S^2} \sum_{\alpha = 2}^n d_{\alpha}^2 u_{\alpha \alpha} + \left[ \sum_{\alpha = 2}^n \frac{\cos \theta d_{\alpha}}{S^2} (\cos \theta u_1 d_{1 \alpha} + K d_{\alpha} v) \right.\\
    & - S ^{- 1} (\cos \theta u_1 d_{11} + K d_1 v)\Bigg]\\
     = &\frac{\cos^2 \theta}{S^2} \sum_{\alpha = 2}^n d_{\alpha}^2 u_{\alpha \alpha} + O (v) .
  \end{aligned} 
  \label{u11}
\end{equation}
  Applying the condition of second derivative on $(x_0, t_0)$, we have
  \begin{equation}
  \begin{aligned}
    0 & \leq \left( \partial_t - Q_c^{i j} \nabla_{i j} {- Q_{c, p_i}} \nabla_i \right) \Phi\\
    & = \frac{(1 + K d)}{v} u_k (u_{k t} - Q_c^{i j} u_{k i j} - Q_{c, p_{\mathit{i}}} u_{k i})\\
    & \quad + d_k \cos \theta (u_{k t} - Q_c^{i j} u_{k i j} - Q_{c, p_i} u_{k i})\\
    & \quad + (1 + K d) \left( \frac{Q_c^{i j} u_l u_{l i} u_k u_{k j}}{v^3} - \frac{Q_c^{i j} u_{l i} u_{l j}}{v} \right)\\
    & \quad - (2 Q_c^{i j} u_{k i} d_{k j} \cos \theta + 2 K Q_c^{i j} d_i v_j)\\
    & \quad - (Q_c^{i j} u_k d_{k i j} \cos \theta + K Q_c^{i j} d_{i j} v)\\
    & \quad - Q_{c, \psi_i} (K d_i v + \cos \theta u_k d_{k i})\\
    & \assign L_1 + L_2 + L_3 + L_4 + L_5 + L_6 .
  \end{aligned} 
  \label{testY}
\end{equation}
  Let us examine the term $L_1$ and $L_2$ at first. Differentiating the
  parabolic equation \eqref{scalflow} with respect to $t$, we have
  \begin{equation}
    u_{t k} = Q_c^{i j} u_{i j k} + Q_{c, p_i} u_{i k} + Q_{c, \rho} \rho u_k
    + Q_{c, \beta} \sigma (\partial_{\beta}, e_k) + Q_{c, \gamma} \sigma
    (\partial_{\gamma}, e_k) . \label{diffut}
  \end{equation}
  Ricci identity on $\mathbb{S}^n_+$ gives that
  \begin{equation}
    u_{i j k} = u_{k i j} + u_j \sigma_{i l} - u_l \sigma_{i j} .
    \label{ricci}
  \end{equation}
  And by definition, we have
\begin{equation}
  \begin{aligned}
    a^{i j} u_{i j} & = \left( \sigma^{i j} - \frac{u^i u^j}{v^2} \right) u_{i j} = u_{11} + \sum_{\alpha = 2}^n u_{\alpha \alpha} - \frac{| \nabla u |^2}{v^2} u_{11}\\
    & = \frac{1}{v^2} u_{11} + \sum_{\alpha = 2}^n u_{\alpha \alpha} .
  \end{aligned} 
  \label{au}
\end{equation}
  Applying \eqref{diffut}, \eqref{ricci} and \eqref{au} to the terms $L_1$ and
  $L_2$ gives,
\begin{equation}
  \begin{aligned}
    L_1 & = \frac{1 + K d}{v} u_k (u_{k t} - Q^{i j} u_{k i j} - Q_{c, p_i} u_{k i})\\
    & = \frac{1 + K d}{v} u_k (Q_c^{i j} (u_j \sigma_{i k} - u_k \sigma_{i j}) + Q_{c, \rho} \rho u_k + Q_{c, \beta} \sigma (\partial_{\beta}, e_k)\\
    & \quad + Q_{c, \gamma} \sigma (\partial_{\gamma}, e_k))\\
    & = \frac{1 + K d}{v^4} \left( \cos \beta \cos \gamma u_{\beta} - \sin \beta \sin \gamma u_{\gamma} - c \frac{| \nabla u |^2}{\rho} \right) u_{11}\\
    & \quad - c \frac{1 + K d}{\rho  v^2} | \nabla u |^2 \left( \sum_{\alpha = 2}^n u_{\alpha \alpha} \right)\\
    & \quad - n c \frac{1 + K d}{v} | \nabla u |^2 \left( \frac{| \nabla u |^2}{\rho v} + \frac{\cos \theta \sin \beta}{\rho} u_{\beta} \right)\\
    & \quad + (1 + K d) \{ \frac{1}{v^2} (\cos \beta \cos \gamma u_{\beta} - \sin \beta \sin \gamma u_{\gamma}) \sum_{\alpha = 2}^n u_{\alpha \alpha}\\
    & \quad + \frac{n \cos \theta \cos \beta}{v \rho } | \nabla u |^2 + \frac{n}{v^2} (\sin \beta \cos \gamma u_{\beta }^2 + 2 \cos \beta \sin \gamma u_{\gamma}  u_{\beta}\\
    & \quad + \sin \beta \cos \gamma u^2_{\gamma}) + \frac{(1 - n) (1 + K d) | \nabla u |^2}{\rho e^w v^2} \}\\
    & \assign L_{11} + L_{12} + L_{13} + L_{14} .
  \end{aligned}\label{J1}
\end{equation}
  It is obvious that $L_{14} = O (v^{- 1}) \sum_{\alpha = 2}^n u_{\alpha
  \alpha} + O (v)$. And using \eqref{u11} we have
\begin{equation}
  \begin{aligned}
    L_{11} & = \frac{1 + K d}{v^4} \left( \cos \beta \cos \gamma u_{\beta} - \sin \beta \sin \gamma u_{\gamma} - c \frac{| \nabla u |^2}{\rho} \right)\left( \frac{\cos^2 \theta}{S^2} \sum_{\alpha = 2}^n d_{\alpha}^2 u_{\alpha \alpha} + O (u_1) \right)\\
    & = O (v^{- 2}) \sum_{\alpha = 2}^n | u_{\alpha \alpha} | + O (v^{- 1}) .
  \end{aligned}
\end{equation}
  Similarly, the term $L_2$ be written as follows,
\begin{equation}
  \begin{aligned}
    L_2 & = d_k \cos \theta (u_{k t} - Q_c^{i j} u_{k i j} - Q_{c, p_i} u_{k i})\\
    & = \frac{\cos \theta}{v^3} \left( \cos \beta \cos \gamma d_{\beta} - \sin \beta \sin \gamma d_{\gamma} - c \frac{\sigma (\nabla u, \nabla d)}{\rho} \right) u_{11}\\
    & \quad - c \frac{\cos \theta}{\rho v } \sigma (\nabla u, \nabla d) \sum_{\alpha = 2}^n u_{\alpha \alpha}\\
    & \quad - n c \sigma (\nabla u, \nabla d) \left( \frac{| \nabla u |^2}{\rho v} + \frac{\cos \theta \sin \beta}{\rho } u_{\beta} \right)\\
    & \quad + \cos \theta \left[ \frac{1}{v} (\cos \beta \cos \gamma d_{\beta} - \sin \beta \sin \gamma d_{\gamma}) \sum_{\alpha = 2}^n u_{\alpha \alpha} \right.\\
    & \quad \quad + \frac{n \cos \theta \cos \beta}{\rho } \sigma (\nabla u, \nabla d) + \frac{n}{v} (\sin \beta \cos \gamma u_{\beta} + \cos \beta \sin \gamma u_{\gamma}) d_{\beta}\\
    & \quad \quad + \frac{n}{v} (\sin \beta \cos \gamma u_{\beta} + \cos \beta \sin \gamma u_{\gamma}) d_{\alpha} + (1 - n) \frac{1}{{\rho e^w} } (\nabla u, \nabla d) \nobracket]\\
    & \assign L_{21} + L_{22} + L_{23} + L_{24}.
  \end{aligned} 
  \label{L2}
\end{equation}
  And for the same reason, we have $L_{21} = O (v^{- 2}) \sum_{\alpha = 2}^n |
  u_{\alpha \alpha} | + O (v^{- 1})$ and $L_{24} = O \left( \frac{1}{v}
  \right) \sum_{\alpha = 2}^n u_{\alpha \alpha} + O (v)$.
  
  For the term $L_3$, we have
\begin{equation}
  \begin{aligned}
    L_3 & = (1 + K d) \left( \frac{Q_c^{i j} u_l u_{l i} u_k u_{k j}}{v^3} - \frac{Q_c^{i j} u_{l i} u_{l j}}{v} \right)\\
    & = \frac{1 + K d}{\rho v e^w} \left( \sigma^{i j} - \frac{u_i u_j}{v^2} \right) \left( \frac{u_l u_{l i} u_k u_{k j}}{v^3} - \frac{u_{l i} u_{l j}}{v} \right)\\
    & = \frac{1 + K d}{\rho v e^w} \left( - \frac{1}{v^5} u_{11}^2 - \frac{2}{v^3} \sum_{\alpha = 2}^n u_{1 \alpha}^2 - \frac{1}{v} \sum_{\alpha = 2}^n u_{\alpha \alpha}^2 \right)\\
    & = \frac{1 + K d}{\rho v e^w} \left( - \frac{1}{v^5} u_{11}^2 - \frac{2}{v^3} \sum_{\alpha = 2}^n u_{1 \alpha}^2 \right) - (1 - \varepsilon) \frac{1 + K d}{\rho v^2 e^w} \sum_{\alpha = 2}^n u_{\alpha \alpha}^2\\
    & \quad - \varepsilon \frac{1 + K d}{\rho v^2 e^w} \sum_{\alpha = 2}^n u_{\alpha \alpha}^2\\
    & \assign L_{31} + L_{32} + L_{33}.
  \end{aligned} 
  \label{J3}
\end{equation}
  Then we compile the following three terms
\begin{equation}
  \begin{aligned}
    L_{12} + L_{22} + L_{32} & = - \frac{c}{\rho  v} \left( \frac{1 + K d}{v} | \nabla u |^2 + \cos \theta \sigma (\nabla u, \nabla d) \right) \sum_{\alpha = 2}^n u_{\alpha \alpha}\\
    & \quad - (1 - \varepsilon) \frac{1 + K d}{\rho v^2 e^w} \sum_{\alpha = 2}^n u_{\alpha \alpha}^{2 }\\
    & = - \frac{c}{\rho  v} S | \nabla u |  \sum_{\alpha = 2}^n u_{\alpha \alpha} - (1 - \varepsilon) \frac{1 + K d}{\rho v^2 e^w} \sum_{\alpha = 2}^n u_{\alpha \alpha}^2\\
    & \leq \frac{c^2 (n - 1) e^w S^2 | \nabla u |^2}{4 \rho  (1 - \varepsilon) (1 + K d)}\\
    & \leq \frac{c^2 (n - 1) (1 + | \cos \theta |) S e^w}{4 \rho (1 - \varepsilon) } | \nabla u  |^2 .
  \end{aligned} 
  \label{est1}
\end{equation}
  To estimate $L_{13} + L_{23}$, we need the following arguments. Let $c_0$ be
  a constant which satisfies $c_0 \in \left( | \cos \theta |, \frac{4 n \tau
  K_0 (c, R, \theta) - c (n - 1)}{4 n \tau K_0 (c, R, \theta) + c (n - 1)}
  \right)$, for some $0 < \tau < 1$. Then we can assume that
  \begin{equation}
    \frac{| \nabla u |^2}{v} + \cos \theta \sin \beta u_{\beta} \geq (1 - c_0)
    | \nabla u | . \label{est2}
  \end{equation}
  Otherwise,
  \begin{eqnarray*}
    1 - c_0 & > & \frac{| \nabla u | }{v} + \cos \theta \sin \beta
    \frac{u_{\beta}}{| \nabla u |}\\
    & \geq & \frac{| \nabla u |}{v} - | \cos \theta |
  \end{eqnarray*}
  would give an upper bound for $| \nabla u |$, the required estimate is
  obtained.
  
  Considering $L_{13} + L_{23}$ and using \eqref{est2}, we have
\begin{equation}
  \begin{aligned}
    L_{13} + L_{23} & = - n c \frac{1 + K d}{v} | \nabla u |^2 \left( \frac{| \nabla u |^2}{\rho v} + \frac{\cos \theta \sin \beta u_{\beta}}{\rho} \right)\\
    & \quad - n c \cos \theta \sigma (\nabla u, \nabla d) \left( \frac{| \nabla u |^2}{\rho v} + \frac{\cos \theta \sin \beta u_{\beta}}{\rho } \right)\\
    & = - n c | \nabla u | \rho^{- 1} \left( \frac{| \nabla u |^2}{v } + \cos \theta \sin \beta u_{\beta} \right) S\\
    & \leq - n c (1 - c_0) \rho^{- 1} S | \nabla u |^2 .
  \end{aligned} 
  \label{est3}
\end{equation}
  Combining \eqref{est1} and \eqref{est3}, we have
\[
\begin{aligned}
   & L_{13} + L_{23} + L_{12} + L_{22} + L_{32} \\
\leq & - n c (1 - c_0) \rho^{-1} S | \nabla u |^2 + \frac{c^2 (n - 1) (1 + | \cos \theta |) S}{4 (1 - \varepsilon) \rho } e^w | \nabla u |^2 \\
\leq & \left( - n c (1 - c_0) \rho^{-1} + \frac{c^2 (n - 1) (1 + c_0)}{4 (1 - \varepsilon) \rho  x_{n + 1}} \right) S | \nabla u |^2 .
\end{aligned}
\]
  For the last inequality, Proposition \ref{C0} shows that $\Sigma_t \subset
  \hat{\mathcal{C}}_{c, R, \theta}$, and hence $e^{- w} = x_{n + 1} > K_0 (c,
  R, \theta)$. Therefore
 \[
\begin{aligned}
   & - n c (1 - c_0) + \frac{c^2 (n - 1) (1 + c_0)}{4 (1 - \varepsilon) x_{n + 1}} \\
\leq & - \frac{2 c^2 n (n - 1)}{4 \tau c n K_0 (c, R, \theta) + (n - 1)} + \frac{2 c^2 \tau n K_0 (c, R, \theta)  (n - 1)}{x_{n + 1} (1 - \varepsilon) (4 \tau c n K_0 (c, R, \theta) + (n - 1))} \\
\leq & - \frac{2 c^2 n (n - 1)}{4 \tau c n K_0 (c, R, \theta) + (n - 1)} \left( 1 - \frac{\tau}{1 - \varepsilon} \right) .
\end{aligned}
\]
  Let $\varepsilon = (1 - \tau) / 2 \in (0, 1)$, from the $C^0$-estimate in
  Proposition \ref{C0}, there is an uniform lower bound $\rho_0$ on $\rho$.
  Hence letting $a_0 = \frac{2 c^2 n (n - 1) (1 - \tau)}{(4 \tau c n K_0 (c,
  R, \theta) + (n - 1)) (1 + \tau)} \rho_0 S$, we have
  \begin{equation}
    L_{13} + L_{23} + L_{12} + L_{22} + L_{32} < - a_0 | \nabla u |^2 .
    \label{est4}
  \end{equation}
  Now we consider $L_4$ and $L_6$. Using \eqref{u11}, \eqref{testY} and
  \eqref{au}, we obtain
\[
\begin{aligned}
   & L_4 + L_6 \\
= & - (2 Q_c^{i j} u_{k i} d_{k j} \cos \theta \nobracket \nobracket + 2 K Q_c^{i j} d_i v_j) - Q_{c, \psi_i} (K d_i v + \cos \theta u_k d_{k i}) \\
= & O \left( \frac{1}{v} \right) \sum_{\alpha = 2}^n | u_{\alpha \alpha} | + O (1),
\end{aligned}
\]

  and it is easy to see $L_5 = - (Q_c^{i j} u_k d_{k i j} \cos \theta + K
  Q_c^{i j} d_{i j} v) = O (1)$.
  
  Then adding all the terms back to \eqref{testY}, we have
 \[
\begin{aligned}
   0 & \leq \frac{1 + K d}{\rho e^w v} \left( - \frac{1}{v^5} u^2_{11} - \frac{2}{v^3} \sum_{\alpha = 2}^n | u_{1 \alpha} | \right) - \varepsilon_0 \frac{1 + K d}{2 \rho e^w v^2} \sum_{\alpha = 2}^n | u_{\alpha \alpha} | - \alpha_0 u_1^2 \\
   & \quad + O \left( \frac{1}{v} \right) \sum_{\alpha = 2}^n | u_{\alpha \alpha} |  + O (v) \\
   & \leq - \varepsilon_0 \frac{1 + K d}{2 \rho e^w v^2} \sum_{\alpha = 2}^n u^2_{\alpha \alpha} + \frac{C_2}{v} \sum_{\alpha = 2}^n | u_{\alpha \alpha} | - a_0 u_1^2 + C_1 v \\
   & \leq - a_0 u_1^2 + C_1 v + \frac{C^2_2 \rho e^w}{2 \varepsilon_0 (1 + K d)},
\end{aligned}
\]
  This gives an upper bound for $u_1$.
\end{proof}

\section{Convergence of the flow}\label{sec:convergence}

The higher order a-priori estimates of $u$ follow from the uniform $C^0$ and
$C^1$ estimates. The same argument as in {\cite{wang2020mean}} gives the
following result.

\begin{proposition}
  \label{Ck}If $u (\cdummy, t)$ solves the boundary value problem
  \eqref{scalflow} on the interval $[0, T^{\ast})$, and the initial
  hypersurfaces $\Sigma_0$ and the contact angle $\theta$ satisfies the same
  condition in Theorem \ref{thm}. Then for any $0 < T < T^{\ast}$, we have
  \[ \| u (\cdummy, t) \|_{C^k} \leq C, \quad 0 < t < T, \]
  where $C = C (k, u_0, \nabla u_0, \nabla^2 u_0) > 0$. In addition, it
  follows that $T^{\ast} = \infty$. 
\end{proposition}

Before we prove Theorem \ref{thm}, we need the following lemma.

\begin{lemma}
  \begin{flushleft}
    Let \label{capiso}$\mathcal{C}_{c_1, R_1, \theta } (a_1)$ and
    $\mathcal{C}_{c_2, R_2, \theta} (a_2)$ defined by \eqref{capbarrier}. If
    their enclosed volume are equal, that is, $| \hat{\mathcal{C}}_{c_1, R_1
    \comma \theta} (a_1) | = | \hat{\mathcal{C}}_{c_2, R_2 \comma \theta}
    (a_2) |$, it holds that
    \[ \frac{c_1}{R_1} = \frac{c_2}{R_2}, \]
    that is, the umbilical $\theta$-caps with the same enclosed volume have
    the same principal curvature.
  \end{flushleft}
\end{lemma}

\begin{proof}
  Let $\chi = \chi_2 \circ \chi_1$ be an isometry composed by two isometries
  in $\mathbb{H}^{n + 1}$, where $\chi_1$ is a translation along the
  hyperbolic geodesic $\gamma_t = (0, \cdots, t) \in \mathbb{R}_+^{n + 1}$,
  such that,
  \[ \chi_1 (x_1, \cdots, x_{n + 1}) = \frac{c_2}{c_1} (x_1, \cdots, x_{n +
     1}), \]
  and $\chi_2$ is a translation defined by
  \[ \chi_2 (\tilde{x}, x_{n + 1}) = (\tilde{x} + a_2 - a_1, x_{n + 1}) . \]
  They are both isometric transformation. Hence
  \[ | \hat{\mathcal{C}}_{c_2, R_2 \comma \theta} (a_2) | = |
     \hat{\mathcal{C}}_{R_1, c_1, \theta} (a_1) | = | \chi
     (\hat{\mathcal{C}}_{R_1, c_1, \theta} (a_1)) | = |
     \hat{\mathcal{C}}_{c_2, c_2 R_1 / c_1, \theta} (a_2) | . \]
  Since $| \hat{\mathcal{C}}_{c_2, R, \theta} (\theta)  |$ is monotonically
  decreasing as $R$ decreasing, so $c_2 R_1 / c_1 = R_2$.
\end{proof}

\begin{remark}
  We can see from the proof that, the umbilical caps $\mathcal{C}_1$ and
  $\mathcal{C}_2$ on totally geodesic hyperplane with the same enclosed volume
  is equivalent to the same area and the same wetting area, which is the area
  of the domain enclosed by $\partial \mathcal{C}_1$ and $\partial
  \mathcal{C}_2$.
\end{remark}

Now we are ready to prove Theorem \ref{thm}.

\begin{proof}[Proof of the Theorem \ref{thm}]
  Let $r = 2$ in the Minkowski type formula \eqref{Mink1}, we have
  \[ \int_{\Sigma_t} \left[ H_1 \left( \frac{c}{x_{n + 1}} - c \cos \theta
     \langle E_1, \nu \rangle \right) - H_2 \langle x - c E_{n + 1}, \nu
     \rangle \right] d A = 0. \]
  Let $\widehat{\partial \Sigma}$ be the domain enclosed by $\partial \Sigma$
  on $P$. The first variation formula of the energy $\mathcal{Q}
  (\hat{\Sigma}_t) = \frac{1}{n} [| \Sigma_t | - \cos \theta |
  \widehat{\partial \Sigma} |]$ gives that
\[
\begin{aligned}
    \frac{d}{d t} \mathcal{Q}(t) & = \int_{\Sigma_t} H_1 \left( \frac{c}{x_{n + 1}} - c \cos \theta \langle E_1, \nu \rangle - H_1 \langle x - c E_{n + 1}, \nu \rangle \right) d A \\
    & = \int_{\Sigma_t} H_1 \left( \frac{c}{x_{n + 1}} - c \cos \theta \langle E_1, \nu \rangle - H_1 \langle x - c E_{n + 1}, \nu \rangle \right) d A \\
    & \quad - \int_{\Sigma_t} \left[ H_1 \left( \frac{c}{x_{n + 1}} - c \cos \theta \langle E_1, \nu \rangle \right) - H_2 \langle x - c E_{n + 1}, \nu \rangle \right] d A \\
    & = \int_{\Sigma_t} (H_1^2 - H_2) \langle x - c E_{n + 1}, \nu \rangle d A \\
    & = - \frac{1}{n^2 (n - 1)} \sum_{i < j} \int_{\Sigma_t} (\kappa_j - \kappa_i)^2 \langle x - c E_{n + 1}, \nu \rangle d A \\
    & = - \frac{1}{n^2 (n - 1)} \sum_{i < j} \int_{\mathbb{S}^n_+} (\kappa_i(y) - \kappa_j(y))^2 dg(y),
\end{aligned}
\]
  where $d g (y) = \frac{\rho e^w}{v} d y$ and $\kappa_i$'s are the principle
  curvatures at $y \in \mathbb{S}_+^n$ identified as a point on $\Sigma_t$.
  Therefore, the energy $\mathcal{Q} (\Sigma) = \tmop{Area} (\Sigma) - \cos
  \theta \mathcal{W} (\partial \Sigma)$ is monotonically decreasing from
  Proposition \ref{C0}, we know that the energy is bounded from above and
  below. Then integrating both sides of the equation above on $[0, + \infty)$,
  we have
  \[ \sum_{i < j} \int_0^{\infty} \int_{\mathbb{S}^n_+} (\kappa_i (y, t) -
     \kappa_j (y, t))^2 d g (y) d t \leq C. \]
  From the uniform $C^k$ estimate in Proposition \ref{Ck}, we have
  \[ \lim_{t \rightarrow \infty} | \kappa_i - \kappa_j |^2 = 0. \]
  Therefore any convergent subsequence of $x (\cdummy, t)$ must converge to an
  umbilical cap $\mathcal{C}_{c, R_{\infty}, \theta} (a_{\infty})$ as $t
  \rightarrow \infty$. Hence from the boundary condition in \eqref{flow}, we
  know that they are given by
  \[ \mathcal{C}_{c', R', \theta} (a_{\infty}) = \{ x \in \mathbb{H}^{n + 1}
     : | x + R' \cos \theta E_1 - a_{\infty} - c' E_{n + 1} | = R'  \}, \]
  where $a_{\infty}$ is a constant vector perpendicular to both $E_1$ and
  $E_{n + 1}$.
  
  It remains to show that the limit umbilical cap is unique. We follow the
  proof in {\cite{mei2023constrained}} and {\cite{wang2023alexandrov}}. Let
  $R_{\infty} > 0$ be the unique number such that $| \hat{\Sigma} | = |
  \hat{\mathcal{C}}_{c, R_{\infty}, \theta} |$.
  
  Denote by $R (\nosymbol \cdummy, t)$ the radius of the unique umbilical cap
  \[ \mathcal{C}_{c, R (\cdummy, t), \theta} = \{ x \in \mathbb{H}^{n + 1} : |
     x - R (\cdummy, t) \cos \theta E_1 - c E_{n + 1}   | = R (\cdummy, t) \},
  \]
  passing through the point $x (\cdummy, t)$. Let $R^{\ast} (t) = \max_{x \in
  M} R (x, t)$ and there exists a point $\xi_t$ attaining $R^{\ast} (t)$ by
  compactness. It then follows that $R^{\ast} (t)$ is non-increasing since the
  $\mathcal{C}_{c, R^{\ast} (\cdummy, t), \theta}$ is a barrier of $x
  (\cdummy, t)$ by Proposition \ref{C0}. We claim that
  \begin{equation}
    \lim_{t \rightarrow \infty} R^{\ast} (t) = R_{\infty} . \label{limitrho}
  \end{equation}
  We suppose otherwise, then there exist $\varepsilon > 0$ and $t$ large
  enough, such that
  \begin{equation}
    R^{\ast} (t) > R_{\infty} + \varepsilon, \label{nonlimitrho}
  \end{equation}
  from the definition of $R (\cdummy, t)$, we have
  \[ | x + R (\cdummy, t) \cos \theta E_1 - c E_{n + 1} |_{\delta}^2 = R^2
     (\cdummy, t), \]
  or equivalently
  \[ | x |_{\delta}^2 + c^2 - 2 c \langle x, E_{n + 1} \rangle_{\delta} - 2 R
     (\cdummy, t) \cos \theta \langle x, E_1 \rangle_{\delta} = R^2 (\cdummy,
     t) \sin^2 \theta .  \]
  Taking the derivative of the equation above with respect to $t$, we have
  \begin{equation}
    \langle x_t, x - c E_{n + 1} + R (\cdummy, t) \cos \theta E_1
    \rangle_{\delta} = R_t (\cdummy, t) (R (\cdummy, t) \sin^2 \theta - \cos
    \theta \langle x, E_1 \rangle_{\delta}) . \label{capderit}
  \end{equation}
  Now we evaluate at the point $ (\xi_t, t)$, since $\Sigma_t$ is tangential
  to $\mathcal{C}_{c, \rho^{\ast} (\cdummy, t), \theta}$ at the point, then
  the normal vector at $(\xi_t, t)$ is
  \[ \tilde{\nu} (\xi_t, t) = \nu_{\mathcal{C}_{c, \rho^{\ast} (t), \theta}}
     (x (\xi_t, t)) = \frac{x + R^{\ast} (t) \cos \theta E_1 - c E_{n +
     1}}{R^{\ast} (t)} . \]

Inserting the flow equation \eqref{flow} to \eqref{capderit}, we have 

\begin{equation}
    \begin{aligned}
        & (R^{\ast}(t) \sin^2 \theta - \cos \theta \langle x, E_1 \rangle_{\delta}) \partial_t R^{\ast}(t) \\
        = & x_{n + 1} \frac{| x + R^{\ast}(t) \cos \theta E_1 - c E_{n + 1} |_{\delta}^2}{R^{\ast}(t)} q_c \\
         = &x_{n + 1} R^{\ast}(t) \left( \frac{n c}{x_{n + 1}} - n c \cos \theta \langle E_1, \nu \rangle - H (\xi_t, t) \langle x - c E_{n + 1}, \nu \rangle \right) .\label{est10}
    \end{aligned}
\end{equation}
  From the calculation in Remark \ref{capstatic}, we have
  \begin{equation}
    H_{\mathcal{C}_{c, \rho^{\ast} (t), \theta}} = \frac{n c}{R^{\ast} (t)},
    \label{H1}
  \end{equation}
  then
  \begin{equation}
    \left. \frac{\frac{n c}{x_{n + 1}} - n c \cos \theta \langle E_1, \nu
    \rangle}{\langle x - c E_{n + 1}, \nu \rangle} \right|_{(\xi_t, t)} =
    \left. \frac{\frac{n c}{x_{n + 1}} - n c \cos \theta \langle E_1, \nu
    \rangle}{\langle x - c E_{n + 1}, \nu \rangle} \right|_{\mathcal{C}_{c,
    \rho^{\ast} (t), \theta}} = \frac{n c}{R^{\ast} (t)} . \label{H2}
  \end{equation}
  \[ \  \]
  From \eqref{supp}, \eqref{H1} and \eqref{H2}, we can see that for
  $\varepsilon>0$ small enough, there exist a $T > 0$ such that for any $t > T$,
  the following
\begin{equation}
    \begin{aligned}
        & x_{n + 1} R^{\ast} (t) \left( \frac{n c}{x_{n + 1}} - n c \cos \theta \langle E_1, \nu \rangle - H (\xi_t, t) \langle x - c E_{n + 1}, \nu \rangle \right) \\
         = & x_{n + 1} (n c - H (\xi_t, t) R^{\ast} (t)) \langle x - c E_{n + 1}, \nu \rangle
    \end{aligned}
    \label{estrhs}
\end{equation}

  holds at $(\xi_t, t)$.
  
  On the other hand, since $x (\cdummy, t)$ converge to $\mathcal{C}_{c ,
  \rho_{\infty}, \theta} (a_{\infty})$ and $\rho_{\infty}$ is uniquely
  determined, we get
  \[ \lim_{t \rightarrow \infty} H (\xi_t, t) = \frac{n c'}{{R }'} . \]
  Therefore, for $\varepsilon_0 = \frac{n c'}{2 R'}$, there exists a $T > 0$,
  such that for any $t > T$,
  \begin{equation}
    H (\xi_t, t) > \frac{n c'}{R'} - \varepsilon_0 > 0. \label{estH}
  \end{equation}
  Then from \eqref{nonlimitrho}, \eqref{estrhs} and \eqref{estH}, we get
\begin{equation}
    \begin{aligned}
        &\;x_{n + 1} R^{\ast} (t) \left( \frac{n c}{x_{n + 1}} - n c \cos \theta \langle E_1, \nu \rangle - H (\xi_t, t) \langle x - c E_{n + 1}, \nu \rangle \right) \\
         \leq &\;\frac{1}{2} x_{n + 1} \left( n c - \left( \frac{n c'}{R'} - \varepsilon_0 \right) (R_{\infty} + \varepsilon) \right) \langle x - c E_{n + 1}, \nu \rangle \\
         \leq &\;\frac{1}{2} x_{n + 1} \left( - \varepsilon_0 R_{\infty} - \varepsilon \left( \frac{n c'}{R'} - \varepsilon_0 \right) \right) \langle x - c E_{n + 1}, \nu \rangle \\
         < &\;0
         \label{capderit2}
    \end{aligned}
\end{equation}
  In the second inequality we have used Proposition $\ref{Ck}$. On the other
  hand, since $0 \leq \langle x, E_1 \rangle_{\delta} \leq 1 - \rho^{\ast} 
  (t) \cos \theta$, we have
  \begin{equation}
    R^{\ast} (t) \sin^2 \theta - \cos \theta \langle x, E_1 \rangle \geq
    R^{\ast} (t) (1 - \cos \theta) > 0, \label{capderit3}
  \end{equation}
  then combining \eqref{est10}, \eqref{capderit2} and \eqref{capderit3}, we
  have
\[ \partial_t R^{\ast} (t) < 0. \]
  This leads to a contradiction to that $\lim_{t \rightarrow 0} \frac{d}{d t}
  R^{\ast} (t) = 0$. Hence \eqref{limitrho} holds. Similarly, we can show
  \[ \lim_{t \rightarrow \infty} R_{\ast} (t) = R_{\infty}, \]
  where $R_{\ast} (t)$ is defined by $R_{\ast} (t) = \min_{x \in M} R (x, t) =
  R (\chi_t, t)$ and $\chi_t$ is the point achieving $R_{\ast} (t)$. Therefore
  $\mathcal{C}_{c', R', \theta} (a_{\infty}) =\mathcal{C}_{c, R_{\infty},
  \theta}$, and we obtain the uniqueness of the limit. 
\end{proof}

\begin{remark}
  In Poincar{\'e} half space model of hyperbolic space, the volume of an
  umbilical cap is determined not only by its Euclidean radius but also by the
  location of its center, particularly the $(n + 1)$-th coordinate. As
  indicated by Remark \ref{capstatic}, the location of its center also
  influences the principal curvatures of the cap. Therefore, we cannot
  identity the radius by the volume of the domain bounded by initial
  hypersurfaces $\Sigma_0$.
  
  Note that all the isometries appearing in the proof of Lemma \ref{capiso} on
  Poincar$\acute{\mathe}$ half-space model will keep the ratio $c / R$ of an
  umbilical cap $\mathcal{C}_{c, R \comma \theta}$. \ 
\end{remark}

From the monotonicity of the energy $\mathcal{Q}$, we have the following
corollary.

\begin{corollary}
  \begin{flushleft}
    \tmcolor{red}{}Let $x : M \rightarrow \mathbb{H}^{n + 1}$ be a
    $\theta$-capillary hypersurface supported on the totally geodesic
    hyperplane $P$. If
    \begin{enumerate}
      \item $\Sigma = x (M)$ is contained in an umbilical cap $\mathcal{C}_{c,
      R, \theta} (a)$ which satisfies that $K (c, R, \theta) > c (n - 1) / 4
      n$.
      
      \item the contacting angle $\theta$ satisfies
      \[ | \cos \theta | < \frac{4 n K_0 (c, R, \theta) - c (n - 1)}{4 n K_0
         (c, R, \theta) + c (n - 1)}, \]
    \end{enumerate}
    then $\theta$-umbilical caps with the same enclosed volume with $\Sigma =
    x (M)$ are the only minimizers of the energy $\mathcal{Q}$. 
  \end{flushleft}
\end{corollary}

\

\printbibliography

\end{document}